\documentclass[11pt]{amsart}
\usepackage[usenames]{color}
\usepackage{fullpage}
\usepackage{amssymb, latexsym}
\usepackage[all,2cell,ps]{xy}
\usepackage{graphicx}

\theoremstyle{plain}
\newtheorem{thm}{Theorem}[section]
\newtheorem{theorem}[thm]{Theorem}

\newtheorem{corollary}[thm]{Corollary}

\newtheorem{lemma}[thm]{Lemma}

\theoremstyle{definition}
\newtheorem{de}[thm]{Definition}

\newtheorem{remark}[thm]{Remark}
\newtheorem{notation}[thm]{Notation}
\newtheorem{example}[thm]{Example}

\newcommand{\Z}{\mathbb{Z}}

\newcommand{\T}{\mathcal{T}}

\numberwithin{equation}{section}

\begin{document}

\title{Knot-theoretic ternary groups}

\author{Maciej Niebrzydowski}
\author{Agata Pilitowska}
\author{Anna Zamojska-Dzienio}

\address{(M.N.) Institute of Mathematics,
Faculty of Mathematics, Physics and Informatics,
University of Gda{\'n}sk, 80-308 Gda{\'n}sk, Poland}
\address{(A.P., A.Z.) Faculty of Mathematics and Information Science, Warsaw University of Technology, Koszykowa 75, 00-662 Warsaw, Poland}

\email{(M.N.) mniebrz@gmail.com}
\email{(A.P.) apili@mini.pw.edu.pl}
\email{(A.Z.) A.Zamojska-Dzienio@mini.pw.edu.pl}



\date{\today}

\begin{abstract}
We describe various properties and give several characterizations of ternary groups satisfying two axioms derived from the third Reidemeister move in knot theory.
Using special attributes of such ternary groups, such as semi-commutativity, we construct a ternary invariant of curves immersed in compact surfaces, considered up to flat Reidemeister moves.
\end{abstract}

\maketitle
\section{Introduction and motivation}
The subject of our paper, knot-theoretic ternary groups, is in the intersection of two classical areas: ternary group theory and knot theory. Ternary groups have a long history. As far back as 1904, E. Kasner considered generalizing the properties of binary groups \cite{Post}. The complete formulation of the concept of $n$-ary groups, generalizing binary groups, appeared in D{\"o}rnte's paper \cite{Dor}. Independently, in 1932, Lehmer \cite{Leh} considered a structure he termed {\it triplex}, which in D{\"o}rnte's terminology is an abelian ternary group. For a thorough introduction to the subject of $n$-ary groups, see \cite{Post}. Knot theory has its roots in the study of
embeddings of simple closed curves in a three-dimensional space $\mathbb{R}^3$ (or $\mathbb{S}^3$). A single such curve is called a {\it knot}, and a collection of knots is a {\it link}. Knots and links are analyzed using {\it diagrams}, that is, projections on the plane that involve double points. Two link diagrams represent the same link if and only if one can be obtained from the other by a finite sequence of {\it Reidemeister moves} of type I, II and III, and planar isotopy. A {\it link invariant} is a function defined on diagrams that is not changed by  Reidemeister moves and planar isotopy. For more details about these notions, and their various generalizations, see for example \cite{Kau}. One of the main goals of knot theory is to find strong and at the same time computable link invariants.
A possible approach to this problem is to use assignments of elements of some binary groupoid $(X,*)$ to the arcs of a given link diagram, simply called {\it colorings}. If the number of colorings is to be unchanged by Reidemeister moves, the binary operation $*$ has to satisfy certain conditions.
In this way one obtains the axioms for racks and quandles, with self-distributivity corresponding to the Reidemeister move of type III, see e.g. \cite{Kam}.

\begin{figure}
\begin{center}
\includegraphics[height=5cm]{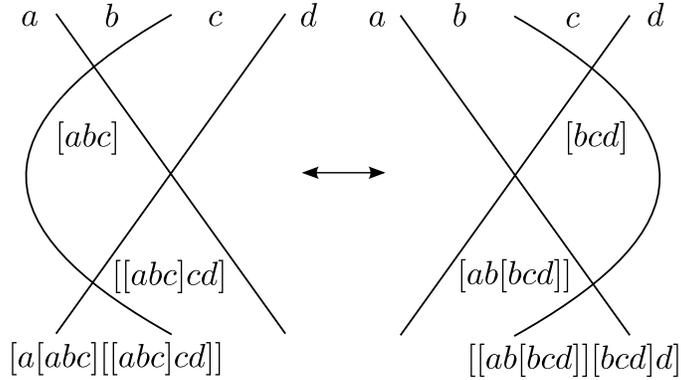}
\caption{Ternary colorings and the third Reidemeister move (flat version)}\label{flatR3}
\end{center}
\end{figure}

Instead of coloring arcs of a link diagram using a binary operation $*$, one can color the regions of the complement of the diagram using a ternary algebra
$(A,[\, ])$. Again, if the number of ternary colorings is to be an invariant, the operation $[\ ]$ has to satisfy some axioms. The first and second Reidemeister move require $(A,[\, ])$ to be a ternary quasigroup. The third Reidemeister move yields two conditions:
\begin{equation}\label{A3L}
[[abc]cd]=[[ab[bcd]][bcd]d],
\end{equation}
\begin{equation}\label{A3R}
[ab[bcd]]=[a[abc][[abc]cd]],
\end{equation}
for all $a$, $b$, $c$, $d\in A$.
They are illustrated in Fig. \ref{flatR3}, where the colors of the regions on the bottom left part of the figure have to be equal to the colors of the corresponding regions on the right, and the flat version of the third Reidemeister move is shown. These axioms appeared in \cite{Nie14}, and in \cite{Nie17} a homology theory was developed for the structures in which they hold. Knot-theoretic ternary groups are ternary groups satisfying \eqref{A3L} and \eqref{A3R}.

The paper is organized as follows. In Section \ref{sec:2} we collect basic results on ternary groups and characterize knot-theoretic ternary groups. In Section \ref{S:Mal} we focus on idempotent ternary groups (all of them are knot-theoretic ones). Section \ref{S:isom} is devoted to representation and isomorphism problems for knot-theoretic ternary groups. In particular, we enumerate such ternary groups with a small number of elements. Section \ref{sec:3} provides further characterization of knot-theoretic ternary groups. Finally, in Section \ref{S:FVL} we construct an invariant of curves immersed in compact surfaces, considered up to flat Reidemeister moves.

\section{Basic properties of knot-theoretic ternary groups}\label{sec:2}
We begin with some preliminary definitions.

A {\it ternary groupoid} is a non-empty set $A$ equipped with a ternary operation $[\, ]\colon A^3\to A$. It is denoted by $(A,[\, ])$.

A ternary groupoid $(A,[\, ])$ is called a \emph{ternary quasigroup} if for every $a,b,c\in A$ each of the following equations is uniquely solvable for $z\in A$:
\begin{align}
&[zab]=c, \label{eq:l}\\
&[azb]=c, \label{eq:m}\\
&[abz]=c.\label{eq:r}
\end{align}

\begin{remark}
Each ternary quasigroup $(A,[\, ])$ has left, middle and right cancellation property. For example, by the unique solvability of the equation \eqref{eq:l} one obtains that for all $x,y,a,b\in A$, $[xab]=[yab]$ always implies that $x=y$ (\emph{left-cancellativity}). \emph{Middle cancellativity} follows by \eqref{eq:m} and \emph{right cancellativity} by \eqref{eq:r}, respectively.
\end{remark}

In a ternary quasigroup $(A,[\, ])$, for every element $a\in A$, the unique solution of the equation $[aza]=a$  is called the \emph{skew element to} $a$ and denoted by $\bar{a}$.

The property of associativity of binary operations can be generalized to ternary case.
We say that an operation $[\, ] \colon A^3 \to A$ is \emph{associative} if for all $a,b,c,d,e\in A$
\begin{align*}
[[abc]de]=[a[bcd]e]=[ab[cde]].
\end{align*}
From this condition follows that in expressions (of length 2$k$-1) obtained by applying the operation $[\, ]$ $k$ times, different placements of parentheses yield the same final result.
An algebra $(A,[\, ])$ with one ternary associative operation is called a \emph{ternary semigroup}.
An associative ternary quasigroup is called a \emph{ternary group}.

If $(A,[\, ])$ is a ternary group, then by \cite[Theorem 10]{BorDuDu} we have for every $a$, $b$, $c\in A$:
\begin{align}
&[\bar{a}aa]=[a\bar{a}a]=[aa\bar{a}]=a, \label{sk:1}\\
&[ba\bar{a}]=[b\bar{a}a]=[a\bar{a}b]=[\bar{a}ab]=b,\label{sk:2}\\
&\overline{[abc]}=[\bar{c}\bar{b}\bar{a}],\label{sk:3}\\
&\bar{\bar{a}}=a.\label{sk:4}
\end{align}
The pair $(a,\bar{a})$ (or $(\bar{a},a)$) plays a role of a \emph{binary identity} \cite{Str}.

A ternary operation $P\colon A^3\to A$ is called a \emph{Mal'cev operation}, if for every $a,b\in A$, $P(a,b,b)=P(b,b,a)=a$. Then, a ternary algebra $(A,P)$ is called a \emph{Mal'cev algebra}. In this way, a ternary group $(A,[\, ])$ is Mal'cev if for all $a,b\in A$
\begin{align*}
[abb]=[bba]=a.
\end{align*}
In general, each ternary group $(A,[\, ])$ has at least one Mal'cev \emph{derived} operation $P(x,y,z)=[x\bar{y}z]$.

We say that a ternary algebra $(A,[\, ])$ is
\begin{itemize}
\item \emph{idempotent} if for every $a\in A$
\begin{align*}
[aaa]=a;
\end{align*}
\item \emph{$(1,2)$-associative} if for all $a,b,c,d,e\in A$
\begin{align*}
[[abc]de]=[a[bcd]e];
\end{align*}
\item \emph{$(2,3)$-associative} if for all $a,b,c,d,e\in A$
\begin{align*}
[a[bcd]e]=[ab[cde]];
\end{align*}
\item \emph{commutative} if for every $a,b,c\in A$
\begin{align*}
[abc]=[bac]=[acb]=[cba];
\end{align*}
\item \emph{semi-commutative} (or \emph{semi-abelian}) if for every $a,b,c\in A$
\begin{align*}
[abc]=[cba];
\end{align*}
\item \emph{entropic} (or \emph{medial}) if for every $a_1,a_2,a_3,b_1, b_2,b_3,c_1,c_2,c_3\in A$
\begin{align*}
[[a_1a_2a_3][b_1b_2b_3][c_1c_2c_3]]=[[a_1b_1c_1][a_2b_2c_2][a_3b_3c_3]].
\end{align*}
\end{itemize}
An element $e\in A$ is called a \emph{left (middle, right) neutral} if for every $a\in A$
\begin{align}\label{eq:neut}
[eea]=a\;\;\ ([eae]=a,\quad [aee]=a).
\end{align}
An element $e\in A$ is called \emph{neutral}, if it is left, middle and right neutral.

\begin{remark}\label{rm:2}
In a semi-commutative ternary groupoid the conditions \eqref{A3L} and \eqref{A3R} are equivalent. Indeed, by \eqref{A3L} and semi-commutativity we have
\begin{align*}
[ab[bcd]]=[[bcd]ba]=[[dcb]ba]=[[dc[cba]][cba]a]=[a[abc][[abc]cd]].
\end{align*}
On the other hand, by \eqref{A3R} and semi-commutativity we obtain
\begin{align*}
[[abc]cd]=[dc[cba]]=[d[dcb][[dcb]ba]]=[[ab[bcd]][bcd]d].
\end{align*}
\end{remark}
\begin{remark}\label{rm:3}
A semi-commutative $(1,2)$-associative ternary groupoid is a ternary semigroup.
\end{remark}
\begin{remark}\label{rm:1}
Each Mal'cev algebra is idempotent. Moreover, a ternary group $(A,[\, ])$ is idempotent if and only if $a=\bar{a}$ for each $a\in A$. In particular, each idempotent ternary group is semi-commutative:
\begin{align*}
[abc]=\overline{[abc]}=[\bar{c}\bar{b}\bar{a}]=[cba].
\end{align*}
\end{remark}
\begin{theorem}\cite[Corollary 9]{Dud80}
A ternary group $(A,[\, ])$ is commutative if and only if for every $a,b\in A$,
\begin{align*}
[ab\bar{a}]=b\;\; or\;\; [\bar{a}ba]=b.
\end{align*}
\end{theorem}

\begin{theorem}\cite[Theorem 3]{GlaGle}
A ternary group is semi-commutative if and only if it is entropic.
\end{theorem}

\begin{lemma}\label{lm:1}
Let $(A,[\, ])$ be a ternary group. Then the following conditions are equivalent for every $a,b,c\in A$
\begin{align}\label{eq:22}
&[abb]=\bar{a}=[bba],
\end{align}

\begin{align}\label{eq:4}
&[\bar{a}bc]=[a\bar{b}c]=[ab\bar{c}].
\end{align}
\end{lemma}
\begin{proof}
By \eqref{eq:22}:
\begin{align*}
[\bar{a}cd]=[[abb]cd]=[a[bbc]d]=[a\bar{c}d]=[a[cdd]d]=[ac[ddd]]=[ac\bar{d}].
\end{align*}
On the other hand, by \eqref{eq:4}:
\begin{align*}
&[abb]=[\bar{\bar{a}}bb]=[\bar{a}\bar{b}b]=\bar{a},\\
&[bba]=[bb\bar{\bar{a}}]=[b\bar{b}\bar{a}]=\bar{a}.
\end{align*}
\end{proof}

\begin{lemma}
Let $(A,[\, ])$ be a semi-commutative ternary group which satisfies for every $a,b\in A$
\begin{align}\label{eq:23}
[abb]=\bar{a}.
\end{align}
Then $(A,[\, ])$ satisfies the conditions \eqref{A3L} and \eqref{A3R}.
\end{lemma}

\begin{proof}
By Remark \ref{rm:2}, it suffices to check only one condition, e.g. \eqref{A3L}. By \eqref{eq:23}, Lemma \ref{lm:1} and semi-commutativity we have
\begin{align*}
&[[ab[bcd]][bcd]d]=[[abb][cdb][cdd]]=[\bar{a}[cdb]\bar{c}]=\\
&[\bar{c}[cdb]\bar{a}]=[db\bar{a}]=[ab\bar{d}]=[ab[ccd]]=[[abc]cd].
\end{align*}

\end{proof}
\begin{de}
A ternary group $(A,[\, ])$ which satisfies \eqref{A3L} and \eqref{A3R} is called a \emph{knot-theoretic ternary group}.
\end{de}

\begin{lemma}\label{cha}
Let $(A,[\, ])$ be a knot-theoretic ternary group.
Then $(A,[\, ])$ is semi-commutative and satisfies \eqref{eq:23}.
\end{lemma}
\begin{proof}

By associativity, \eqref{A3L}, \eqref{sk:1}, middle and right cancellation properties we obtain
\begin{align*}
&[ab[ccc]]=[[abc]cc]=[[ab[bcc]][bcc]c]=[a[b[bcc]b][ccc]]\; \Rightarrow\;\\
&b=[b[bcc]b]=[bb[ccb]] \; \Rightarrow\; [ccb]=\bar{b}.
\end{align*}

Similarly, by associativity, \eqref{A3R}, \eqref{sk:1}, middle and left cancellation properties we obtain
\begin{align*}
[[bbb]cd]=[bb[bcd]]=[b[bbc][[bbc]cd]]=[[bbb][c[bbc]c]d]\; \Rightarrow\; c=[c[bbc]c]=[[cbb]cc] \; \Rightarrow\; [cbb]=\bar{c}.
\end{align*}
Finally, by Lemma \ref{lm:1}
\begin{align*}
[abc]=[\overline{[abc]}aa]=[[\bar{c}\bar{b}\bar{a}]aa]=[\bar{c}\bar{b}[\bar{a}aa]]=
[\bar{c}\bar{b}a]=[cba],
\end{align*}
which means that $(A,[\, ])$ is semi-commutative.
\end{proof}

\begin{corollary}\label{cor:knot}
Let $(A,[\, ])$ be a ternary group. $(A,[\, ])$ is a knot-theoretic ternary group
if and only if
$(A,[\, ])$ is semi-commutative and satisfies \eqref{eq:23}. In particular, in each knot-theoretic group
\begin{align*}
[ccc]=\bar{c},
\end{align*}
for every $c\in A$.
\end{corollary}

\begin{example}\label{ex:Zk}
Let $k>1$ be a natural number. Consider a cyclic group $(\Z_k,+)$ and let $a$ be a fixed element of $\Z_k$. Define on the set $\Z_k$ the ternary operation $[xyz]=x-y+z+a\pmod k$. Then $(\Z_k,[\,])$ is a knot-theoretic ternary group if and only if $2a=0 \pmod k$ in $\Z_k$. For each even $k$, there are exactly two knot-theoretic groups constructed in this way: \begin{enumerate}
\item idempotent one for $a=0$,
\item non-idempotent for $a$ being the element of order 2 in $\Z_k$.
\end{enumerate}
\end{example}
\section{Idempotent ternary groups}\label{S:Mal}

By results of Section \ref{sec:2} one obtains the relation between idempotent (knot-theoretic) ternary groups and Mal'cev ones.
\begin{corollary} \label{idMalkt}
Let $(A,[\, ])$ be a ternary group. The following conditions are equivalent:
\begin{enumerate}
\item  $(A,[\, ])$ is idempotent,
\item $(A,[\, ])$ is idempotent knot-theoretic,
\item $(A,[\, ])$ is Mal'cev.
\end{enumerate}
\end{corollary}
But it is also possible to obtain an idempotent (knot-theoretic) ternary group from any semi-commutative ternary group. The aim of this Section is to provide details of such construction.

\begin{lemma} \label{multiex}
Let $(A,[\, ])$ be a ternary group, and define a new (derived) ternary operation by
\[
P(a,b,c)=[a\bar{b}c],
\]
for $a,b,c\in A$. Then $(A,P)$ is a Mal'cev algebra which satisfies for every $a, b, c, d\in A$ the conditions
\eqref{A3L} and \eqref{A3R}.
\end{lemma}
\begin{proof}
It is easy to show that in the algebra $(A,P)$ the following identities hold
\begin{align}
P(P(x,y,z),z,t)&=P(x,y,t),\label{eq:1M}\\
P(x,y,P(y,z,t))&=P(x,z,t)\label{eq:2M},
\end{align}
for all $x,y,z,t\in A$.
\end{proof}
In fact, the identities \eqref{eq:1M}, \eqref{eq:2M} are satisfied in the algebra $(A,P)$ independently from the order of variables in the definition of the operation $P$, e.g. one can also take $P(a,b,c)=[b\bar{a}c]$, etc.

\begin{lemma}\label{lm:semi}
Let $(A,[\, ])$ be a ternary group, and let
\[
P(a,b,c)=[a\bar{b}c].
\]
Then $(A,[\, ])$ is semi-commutative if and only if $(A,P)$ is a ternary semigroup.
\end{lemma}
\begin{proof}
For  every $a,b,c,d,e\in A$ we have:
\begin{align*}
P(P(a,b,c),d,e) &=[[a\bar{b}c]\bar{d}e]=[a\bar{b}[c\bar{d}e]]=P(a,b,P(c,d,e)).
\end{align*}
Moreover,
\begin{align*}
P(a,P(b,c,d),e) &=P(P(a,b,c),d,e) \; \Leftrightarrow \;
[a\overline{[b\bar{c}d]}e]=[[a\bar{b}c]\bar{d}e]\; \Leftrightarrow \; \\
[a[\bar{d}c\bar{b}]e]&=[a[\bar{b}c\bar{d}]e]\; \Leftrightarrow \;[\bar{d}c\bar{b}]=[\bar{b}c\bar{d}].
\end{align*}
\end{proof}

\begin{corollary}

Let $(A,[\, ])$ be a semi-commutative ternary group, and let
\[
P(a,b,c)=[a\bar{b}c].
\]
Then $(A,P)$ is a Mal'cev ternary group.

\end{corollary}
\begin{proof}
Since $\bar{\bar{x}}=x$, the map $x\mapsto\bar{x}$ is $1-1$, and the quasigroup property of $(A,P)$ follows from the fact that $(A,[\, ])$ is a ternary quasigroup. By Lemma \ref{lm:semi}, the ternary quasigroup $(A,P)$ is a ternary group.
\end{proof}

\begin{example}\label{Zkcont_1}
Consider the two types of knot-theoretic ternary groups introduced in Example \ref{ex:Zk}.
For an even $k$, we take the cyclic group $(\Z_k,+)$ with a ternary operation:
\[
[xyz]=x-y+z+\frac{k}{2} \pmod k,
\]
and a unary operation
\[
\bar{y}=y+\frac{k}{2} \pmod k
\]
for any $x$, $y$, $z\in \Z_k$.
Now define
\[
P(x,y,z)=x-\bar{y}+z+\frac{k}{2}=x-y+z \pmod k,
\]
and we obtain the idempotent ternary group $(\Z_k,P)$.
\end{example}
\section{Isomorphisms of knot theoretic ternary groups}\label{S:isom}

The aim of this section is to recognize when two knot theoretic ternary groups are isomorphic. We start with formulating representation theorems for 
knot theoretic ternary groups.
\begin{de}
Let $(A,[\,])$ be a ternary algebra.
A binary groupoid $(A,*)$, where $x*y=[xay]$ for some fixed $a\in A$, is called a
{\it retract} of $(A,[\,])$, and is denoted by $ret_a(A,[\,])$.
\end{de}

\begin{remark}\label{rm:DG}
It is well known (see e.g. \cite{DudGla}) that a retract $ret_a(A,[\,])$ of a semi-commutative ternary group is an abelian group with the neutral element $\bar{a}$ and the inverse of $x\in A$ given by $\overline{[axa]}$.
\end{remark}
\begin{lemma}\cite[Corollary 15]{BorDuDu} \label{simpdescr}
Let $(A,[\, ])$ be a semi-commutative ternary group and let $a$ be a fixed element of $A$. Then for an abelian group $(A,+)=ret_a(A,[\,])$ and its involutive automorphism $\phi(x)=[\bar{a}xa]$, one has
\[
[xyz]=x+\phi(y)+z+b,
\]
where $b=[\bar{a}\bar{a}\bar{a}]$.
\end{lemma}

\begin{theorem}\label{cor:inv}
Let $(A,[\, ])$ be a ternary group. The following conditions are equivalent
\begin{enumerate}
\item$(A,[\, ])$ is a knot-theoretic ternary group;
\item there is an abelian group
$(A,+)$ with the neutral element  $e\in A$, such that
\begin{align*}
[xyz]=x-y+z+\bar{e},
\end{align*}
and the mapping \, $\bar{}\colon A\longrightarrow A$ is an involutive bijection such that for every $x\in A$, $\bar{x}=x+\bar{e}$.

\end{enumerate}
\end{theorem}

\begin{proof}
Let $(A,[\, ])$ be a knot-theoretic ternary group. Then it is enough to take $(A,+)=ret_{\bar{e}}(A,[\, ])$, which, by Remark \ref{rm:DG}, is an abelian group, with the neutral element $\bar{\bar{e}}=e$, and $-x=[e\bar{x}e]$.
Clearly, $x+\bar{e}=[x\bar{e}\bar{e}]=\bar{x}$. The statement follows because
$\phi(y)=[\bar{\bar{e}}y\bar{e}]=[e\bar{y}e]=-y$ and $b=[eee]=\bar{e}$.

Now let $(A,+)$ be an abelian group with the neutral element  $e\in A$ and an involutive bijection $\bar{}\colon A\longrightarrow A$ such that
for every $x\in A$, $x+\bar{e}=\bar{x}$. It is evident that $(A,[\,])$ with $[xyz]=x-y+z+\bar{e}$, is a knot-theoretic ternary group.
\end{proof}
\begin{corollary}\label{cor:count}
Each knot-theoretic ternary group $(A,[\, ])$ is determined by an abelian group $(A,+)$ and an element $a\in A$ of order one or two in $(A,+)$.
Then for every $x,y,z\in A$
\[
[xyz]=x-y+z+a\quad {\rm and}\quad \bar{x}=x+a.
\]
In particular, for each $x,y\in A$
\begin{align}\label{eq:izo}
[xy\bar{x}]=x-y+\bar{x}+a=x-y+x+a+a=2x-y.
\end{align}
\end{corollary}
\begin{notation}
We denote the knot-theoretic ternary group $(A,[\, ])$ described in Corollary \ref{cor:count} by $\T((A,+),a)$ and we always assume that $(A,+)$ is an abelian group and $a$ is its fixed element of order one or two. Obviously, $ret_a(\T((A,+),a))=(A,+)$. The group $(A,+)$ is called the \emph{associated group} of $(A,[\, ])=\T((A,+),a)$.
\end{notation}
\begin{lemma}\label{lem:iz}
Let $(A,[\,])=\T((A,+),a)$ be a knot-theoretic ternary group. Then for arbitrary $b\in A$, the groups
$ret_b(A,[\, ])$ and $(A,+)$ are isomorphic.
\end{lemma}
\begin{proof}
The bijection $h\colon A\to A$; $x\mapsto x-b+a$ is an isomorphism of groups $(A,\ast)=ret_b(A,[\, ])$ and $(A,+)$. Indeed, for $x,y\in A$
\[
h(x*y)=h([xby])=h(x-b+y+a)=x-b+y+a-b+a=(x-b+a)+(y-b+a)=h(x)+h(y),
\]
which completes the proof.
\end{proof}

\begin{de} A ternary group $(A,[\, ])$ is $g$-\emph{derived from} a binary group $(A,*)$ if for every $x,y,z\in A$, $[xyz]=[xyz]_g:=x*y*z*g$, for some $g\in A$. If $g$ is a neutral element in $(A,*)$, then we simply say that $(A,[\, ])$ is \emph{derived from} (or \emph{reducible to}) $(A,*)$.
\end{de}
Let $(A,[\,])=\T((A,+),a)$ be a commutative knot-theoretic ternary group. Then in the associated abelian group $(A,+)$,
\[-x=[e\bar{x}e]=[ex\bar{e}]=[xe\bar{e}]=x,
\]
for every $x\in A$.
This immediately implies
\begin{corollary}
Let $(A,[\,])=\T((A,+),a)$ be a commutative knot-theoretic ternary group. The associated group $(A,+)$
is an elementary 2-group (Boolean group). Moreover, $(A,[\, ])$ is $a$-derived from $(A,+)$.
\end{corollary}

Recall  that every elementary 2-group is abelian and is a direct sum of cyclic groups of order 2. Hence, in finite case, every such group is isomorphic to the group $\Z_2^k$, for some natural number $k$.
\begin{corollary}\label{cor:fin}
Each finite commutative knot-theoretic ternary group is $a$-derived from the group $\Z_2^k$, for some natural number $k$ and some $a\in\Z_2^k$. In particular, each finite commutative knot-theoretic ternary group has $2^k$ elements for some natural number $k$.
\end{corollary}

\begin{lemma}\label{lem:Bool}
Let $(A,[\, ])$ be a ternary group derived from a group $(A,*)$. Then $(A,[\, ])$ satisfies \eqref{eq:23} if and only if $(A,*)$ is an elementary 2-group.
\end{lemma}
\begin{proof}
Let $(A,[\, ])$ be a ternary group derived from a group $(A,*)$. By \cite{DudGla} in a ternary group derived from a binary group $(A,*)$, the skew element of $a$ coincides with the inverse element in $(A,*)$. Let $e\in A$ be the neutral element in $(A,*)$. Let $(A,[\, ])$ satisfy \eqref{eq:23}, then for every $x\in A$
\begin{align*}
x^{-1}=\bar{x}=[xee]=x*e*e=x\;\; \Rightarrow\;\; x^2=e.
\end{align*}
If $(A,*)$ is an elementary 2-group, then
\begin{align*}
[abb]=a*b*b=a*e=a=a^{-1}=\bar{a}.
\end{align*}
\end{proof}

\begin{corollary}\label{commid}
The following conditions are equivalent:
\begin{enumerate}
\item $(A,[\, ])$ is a commutative ternary Mal'cev group;
\item $(A,[\, ])$ is a knot-theoretic ternary group derived from a binary group $(A,*)$;
\item $(A,[\, ])$ is a knot-theoretic ternary group derived from an elementary 2-group.
\end{enumerate}
\end{corollary}

\begin{example}
The ternary group $(\Z_2,[\, ])$ with $[xyz]=x+ y+ z\pmod 2$ is the only two-element knot-theoretic ternary group derived from a binary group. The knot-theoretic ternary groups from Example \ref{ex:Zk} for $k>2$ are not derived from any binary group.
\end{example}

Having the representation theorems for knot-theoretic ternary groups, we can move to isomorphism problems. We start with recalling a general definition of a homomorphism of ternary groups.

\begin{de}\label{def:hom}
Let $(A_1,[\, ]_1)$ and $(A_2,[\, ]_2)$ be ternary groups with unary operations denoted by $\bar{\;\;}^{1}$ and $\bar{\;\;}^{2}$, respectively. A mapping $h\colon A_1\rightarrow A_2$ is a \emph{homomorphism} of ternary groups, if $h$ \emph{preserves} both fundamental operations, i.e.
\begin{align}
h([abc]_1)&=[h(a)h(b)h(c)]_2,\label{hom:1}\\
h(\bar{a}^{1})&=\overline{h(a)}^{2}\label{hom:2}
\end{align}
for every $a,b,c\in A_1$. A \emph{bijective} homomorphism is an \emph{isomorphism} and an isomorphism from $(A_1,[\, ]_1)$ to itself is an \emph{automorphism}.
\end{de}

Note that the property \eqref{hom:2} follows by \eqref{hom:1} and the definition of a skew element. More precisely, one has
\begin{equation*}
h(a)=h([a\bar{a}^{1}a]_1)=[h(a)h(\bar{a}^{1})h(a)]_2,
\end{equation*}
so $h(\bar{a}^{1})$ must be a skew element to $h(a)$. This means that \eqref{hom:1} suffices to define a homomorphism of ternary groups.

The following result concerning isomorphisms between ternary groups was proved in \cite{DudMi}.
\begin{theorem}\cite[Corollary 7]{DudMi}\label{th:izo}
The following statements are equivalent for ternary groups $(A_1,[\, ]_1)$ and $(A_2,[\, ]_2)$:
\begin{enumerate}
\item  $(A_1,[\, ]_1)$ and $(A_2,[\, ]_2)$ are isomorphic,
\item for every $c\in A_1$ there exists a group isomorphism $h\colon ret_c(A_1,[\, ]_1)\to ret_{d}(A_2,[\, ]_2)$, such that $d=h(c)$,  $h([\bar{c}\bar{c}\bar{c}]_1)=[\bar{d}\bar{d}\bar{d}]_2$ and   $h([cx\bar{c}^1]_1)=[dh(x)\overline{d}^2]_2$ for each $x\in A_1$,
\item for some $c\in A_1$ there exists a group isomorphism $h\colon ret_c(A_1,[\, ]_1)\to ret_{d}(A_2,[\, ]_2)$, such that $d=h(c)$,  $h([\bar{c}\bar{c}\bar{c}]_1)=[\bar{d}\bar{d}\bar{d}]_2$ and $h([cx\bar{c}^1]_1)=[dh(x)\overline{d}^2]_2$ for each $x\in A_1$.
\end{enumerate}
\end{theorem}

Lemma \ref{lem:iz} and the conditions \eqref{eq:23} and \eqref{eq:izo} satisfied in each knot-theoretic ternary group
simplify Theorem \ref{th:izo}. 
Now it becomes:
\begin{theorem}\label{cor:izo}
Let $(A_1,[\, ]_1)=\T((A_1,+_1),a)$ and $(A_2,[\, ]_2)=\T((A_2,+_2),b)$ be two knot-theoretic ternary groups. Then the following statements are equivalent:
\begin{enumerate}
\item  $(A_1,[\, ]_1)$ and $(A_2,[\, ]_2)$ are isomorphic;
\item there exists a group isomorphism $h\colon (A_1,+_1)\to (A_2,+_2)$ such that $b=h(a)$.
\end{enumerate}
\end{theorem}

In particular, using Corollary \ref{cor:fin}, we can characterize all non-isomorphic finite commutative knot-theoretic ternary groups.
\begin{corollary}\label{cor:izocom}
Let $k$ be a natural number and $a,b\in \Z_2^k$. Two commutative knot-theoretic ternary groups $(\Z_2^k,[\,]_a)$ and $(\Z_2^k,[\,]_b)$
are isomorphic if and only if there is a group automorphism $h\colon \Z_2^k\to \Z_2^k$ such that $b=h(a)$.
\end{corollary}

\subsection*{Enumerating knot-theoretic ternary groups}
We can apply the above results to enumerate isomorphism classes of knot-theoretic ternary groups. By Theorem \ref{cor:izo}, knot-theoretic ternary groups associated with non-isomorphic abelian groups are non-isomorphic. Moreover, if there is a group automorphism $h\colon (A,+)\to (A,+)$, then two knot-theoretic ternary groups $(A,[\,])=\T((A,+),a)$ and $(A,<\,>)=\T((A,+),h(a))$
are isomorphic.

\begin{example}
For $(A,+)=(\Z_2\times \Z_4,+)$, there are three non-isomorphic knot-theoretic ternary groups: $\T((A,+),(0,0))$, $\T((A,+),(0,2))$ and $\T((A,+),(1,0))$ (isomorphic to $\T((A,+),(1,2))$.
\end{example}

For each abelian group there is exactly one idempotent knot-theoretic ternary group. 
By Corollary \ref{cor:count}, these are the only knot-theoretic ternary groups with odd number of elements.

Further, it is known that for each $k\in \mathbb{N}$ and two non-neutral elements $a\neq b\in \Z_2^k$ there is a group automorphism
$h\colon \Z_2^k\to \Z_2^k$ such that $b=h(a)$. Then by Corollaries \ref{cor:fin} and \ref{cor:izocom}, for each $k\in \mathbb{N}$
there are exactly two non-isomorphic commutative knot-theoretic ternary groups: idempotent $(\Z_2^k,[\,]_e)$ and non-idempotent $(\Z_2^k,[\,]_a)$, where $e$ is the neutral element in the group $(\Z_2^k,+)$ and $a$ is an arbitrary element in $\Z_2^k-\{e\}$. In Table 1 we listed the numbers of knot-theoretic ternary groups of size $\leq 64$ up to isomorphism.

\begin{table}
$$\begin{array}{|r|cccccccccccccccccc|}\hline
n & 1&2&3&4&5&6&7&8&9&10&11&12&13&14&15&16&17&18\\\hline
\text{all}& 1& 2& 1& 4& 1& 2& 1& 7& 2& 2&1&4&1&2&1&12&1&4\\
\text{idempotent}&1 & 1& 1& 2& 1& 1& 1& 3& 2& 1&1&2&1&1&1&5&1&2\\
\text{commutative}&1 & 2& 0& 2& 0& 0& 0& 2& 0& 0& 0&0&0&0&0&2&0&0\\\hline
\end{array}$$

$$\begin{array}{|r|ccccccccccccccc|}\hline
n & 19&20&21&22&23&24&25&26&27&28&29&30&31&32&33\\\hline
\text{all}&1 &4&1&2&1& 7& 2& 2& 3& 4& 1& 2&1 & 19&1\\
\text{idempotent}&1&2&1&1&1& 3& 2& 1& 3& 2& 1& 1& 1& 7& 1\\
\text{commutative}&0 &0&0&0&0& 0& 0& 0& 0& 0& 0& 0& 0& 2& 0\\\hline
\end{array}$$

$$\begin{array}{|r|ccccccccccccccc|}\hline
n & 34&35&36&37&38&39&40&41&42&43&44&45&46&47&48\\\hline
\text{all}&2&1 &  10& 1&2 & 1& 7& 1&2 & 1&4&2&2&1&10\\
\text{idempotent}&1&1 & 5& 1& 1&1 & 3& 1& 1&1 & 2&2&1&1&4\\
\text{commutative}&0& 0& 0& 0& 0& 0& 0& 0& 0& 0& 0& 0&0&0&0\\\hline
\end{array}$$

$$\begin{array}{|r|cccccccccccccccc|}\hline
n & 49&50&51&52&53&54&55&56&57&58&59&60&61&62&63&64\\\hline
\text{all}&2&4 &  1& 4&1 & 6& 1& 7& 1& 2&1&4&1&2&2&30\\
\text{idempotent}&2&2 & 1& 2& 1&3 & 1&3 & 1& 1&1 &2&1&1&2&11\\
\text{commutative}&0& 0& 0& 0& 0& 0& 0& 0& 0& 0& 0& 0&0&0&0&2\\\hline
\end{array}$$
\caption{The number of knot-theoretic ternary groups of size $n$, up to isomorphism.}
\label{Fig:count_all}
\end{table}

\section{Further characterization of knot-theoretic ternary groups}\label{sec:3}
Now we present some characterizations of knot-theoretic ternary groups using sets of properties that do not include associativity and unique-solvability.
\begin{lemma}\cite[Corollary 4, for $n=3$]{DudGla}\label{lm:cor4}
A ternary semigroup $(A,[\, ])$ is a ternary group if and only if for all $a,b\in A$ there are $x,y\in A$ such that
\begin{align}\label{eq:cor4}
[xaa]=[aay]=b.
\end{align}
\end{lemma}
\begin{lemma}\label{lm:12assemi}
Let $(A,[\, ])$ be a semi-commutative, $(1,2)$-associative ternary groupoid. Let a unary operation $\bar{\ }\colon x\mapsto \bar{x}$ exist on $A$, such that for all $a,b \in A$, the conditions \eqref{eq:23} and
\begin{align}\label{a3a}
[\bar{a}ab]=b
\end{align}
are satisfied.
Then $(A,[\, ])$ is a knot-theoretic ternary group.
\end{lemma}
\begin{proof}
By Remark \ref{rm:3}, $(A,[\, ])$ is a ternary semigroup. By \eqref{eq:23} and \eqref{a3a},
\[
\bar{\bar{a}}=[\bar{a}aa]=a.
\]
Thus, by semi-commutativity and Lemma \ref{lm:cor4} (substituting $x=y=\bar{b}$ in \eqref{eq:cor4}), $(A,[\, ])$ is a ternary semigroup. By Corollary \ref{cor:knot} it is knot-theoretic.
\end{proof}

\begin{theorem} \label{exchangesthm}
A ternary groupoid $(A,[\, ])$ is a knot-theoretic ternary group if and only if it satisfies, for all $x,a,b,c,d,e\in A$, the following conditions:
\begin{align}\label{a1}
[[abc]de]=[[adc]be],
\end{align}
\begin{align}\label{a2}
[a[bcd]e]=[a[bed]c],
\end{align}
and there exists on $A$ a unary operation $\bar{\ }\colon x\mapsto\bar{x}$ such that \eqref{eq:23} and
\begin{align}\label{a3}
[\bar{a}ax]=[xa\bar{a}]=x,
\end{align}
are satisfied.
\end{theorem}
\begin{proof}
Suppose that $(A,[\, ])$ is a knot-theoretic ternary group. Then by Corollary \ref{cor:knot} it is semi-commutative and satisfies the conditions \eqref{a3} and \eqref{eq:23}. From associativity and semi-commu\-ta\-ti\-vi\-ty follow the exchange conditions \eqref{a1} and \eqref{a2}.

Now let $(A,[\, ])$ satisfy \eqref{a1}, \eqref{a2}, \eqref{a3} and \eqref{eq:23}.
By \eqref{eq:23} and \eqref{a3}, as in the proof of Lemma \ref{lm:12assemi}, $\bar{\bar{a}}=a$.
Thus, from \eqref{a3},
\[
[a\bar{a}x]=[\bar{\bar{a}}\bar{a}x]=x,
\]
and
\[
[x\bar{a}a]=[x\bar{a}\bar{\bar{a}}]=x.
\]
From \eqref{a1} follows semi-commutativity:
\[
[abc]=[[c\bar{c}a]bc]=[[cba]\bar{c}c]=[cba].
\]
Now, from \eqref{a1} and \eqref{a2}, we obtain $(1,2)$-associativity:
\[
[a[bcd]e]=[[ab\bar{b}][bcd]e]=[[a[bcd]\bar{b}]be]=[[a[b\bar{b}d]c]be]=
[[adc]be]=[[abc]de].
\]
By Lemma, \ref{lm:12assemi} $(A,[\, ])$ is a knot-theoretic ternary group.
\end{proof}

\begin{theorem}
A ternary groupoid $(A,[\, ])$ is a knot-theoretic ternary group if and only if it satisfies the conditions \eqref{a3}, \eqref{eq:23} and is entropic.
\end{theorem}
\begin{proof}
If $(A,[\, ])$ is a knot-theoretic ternary group then, by Corollary \ref{cor:knot}, it is a semi-commutative ternary group, and therefore it is entropic. Also, from the same corollary, it satisfies \eqref{eq:23}, and each ternary group satisfies \eqref{a3}.

Now suppose that an entropic ternary groupoid $(A,[\, ])$ satisfies
\eqref{a3} and \eqref{eq:23}. Then from \eqref{eq:23} and \eqref{a3}, as in the proof of Theorem \ref{exchangesthm}, we obtain
$[a\bar{a}x]=[x\bar{a}a]=x$, for all $a$, $x\in A$. From this and from the entropic property follows
$(1,2)$-associativity, and semi-commutativity:
\[
[[abc]de]=[[abc][dc\bar{c}][\bar{d}de]]=[[ad\bar{d}][bcd][c\bar{c}e]]=[a[bcd]e],
\]
\[
[abc]=[[ab\bar{b}][\bar{a}ab][c\bar{a}a]]=[[a\bar{a}c][ba\bar{a}][\bar{b}ba]]=[cba].
\]
By Lemma \ref{lm:12assemi}, $(A,[\, ])$ is a knot-theoretic ternary group.
\end{proof}

\begin{theorem}
A ternary groupoid $(A,[\, ])$ is a knot-theoretic ternary group derived from a binary group if and only if it satisfies the condition \eqref{a1} and each element in $A$ is neutral.
\end{theorem}
\begin{proof}
If $(A,[\, ])$ is a knot-theoretic ternary group derived from a binary group, then it is commutative and idempotent ternary group and
by Corollary \ref{cor:knot} it satisfies \eqref{eq:23}. This implies that \eqref{a1} follows and for every $a,b\in A$, $a=\bar{a}$ and
\begin{align}\label{a5}
[abb]=[bab]=[bba]=a.
\end{align}
Now, let $(A,[\, ])$ be a ternary groupoid satisfying \eqref{a1} and \eqref{a5}.
Then $(A,[\, ])$ is idempotent and from \eqref{a1}, follows semi-commutativity:
\[
[abc]=[[cca]bc]=[[cba]cc]=[cba].
\]
Further, from semi-commutativity we get 
\[
[abc]=[[bab]bc]=[[bbb]ac]=[bac]=[cab]=[[aca]ab]=[acb],
\]
and we obtain that $(A,[\, ])$ is commutative.
Now we show
$(1,2)$-associativity using commutativity and \eqref{a1}:
\[
[[abc]de]=[[bac]de]=[[bdc]ae]=[a[bdc]e]=[a[bcd]e].
\]
As before, by Lemma \ref{lm:12assemi}, $(A,[\, ])$ is a knot-theoretic ternary group. From Corollary \ref{commid} follows that $(A,[\, ])$ is derived from a binary group.
\end{proof}
\begin{theorem}
A ternary groupoid $(A,[\, ])$ is a knot-theoretic ternary group if and only if it is semi-commutative, satisfies conditions \eqref{A3L} and \eqref{a2}, and there exists on $A$ an involutive unary operation $\bar{\ }\colon x\mapsto\bar{x}$ such that $[a\bar{a}x]=x$, for all $a$, $x\in A$.
\end{theorem}
\begin{proof}
If $(A,[\, ])$ is a knot-theoretic ternary group, then by definition it satisfies
\eqref{A3L}. The skew element operation is involutive and satisfies $[a\bar{a}x]=x$, for all $a$ and $x\in A$. Property \eqref{a2} follows from associativity and semi-commutativity.

Now suppose that a ternary groupoid $(A,[\, ])$ satisfies the properties from the second part of the theorem. Then semi-commutativity gives
\[
[xa\bar{a}]=[\bar{a}ax]=[\bar{a}\bar{\bar{a}}x]=x,
\]
and
\[
[x\bar{a}a]=[a\bar{a}x]=x.
\]
By Remark \ref{rm:2} the second knot-theoretic property \eqref{A3R} follows from \eqref{A3L}. Now as consequence of \eqref{A3R} for $c=\bar{b}$ and $d=b$ we obtain for any $a,b\in A$:
\begin{align*}
[abb]=[ab[b\bar{b}b]]=[a[ab\bar{b}][[ab\bar{b}]\bar{b}b]]=[aaa]\quad \Rightarrow\quad [bba]=[abb]=[a\bar{a}\bar{a}]=\bar{a}.
\end{align*}
\noindent
From \eqref{A3L} for $c=b$ follows that for $a,b,d\in A$
\[
[\bar{a}bd]=[[abb]bd]=[[ab[bbd]][bbd]d]=[[ab\bar{d}]\bar{d}d]=[ab\bar{d}].
\]
From \eqref{a2} we have
\[
[\bar{a}bc]=[ab\bar{c}]=[a[\bar{b}cc]\bar{c}]=[a[\bar{b}\bar{c}c]c]=[a\bar{b}c].
\]
Thus,
\[
[\bar{a}\bar{b}c]=[a\bar{\bar{b}}c]=[abc].
\]
Finally, we obtain $(2,3)$-associativity:
\[
[ab[cde]]=[a[\bar{b}cc][cde]]=[a[\bar{b}[cde]c]c]=[a[\bar{b}[cce]d]c]=
[a[\bar{b}\bar{e}d]c]=[a[bed]c]=[a[bcd]e].
\]
From $(2,3)$-associativity and property \eqref{a2}, we get $(1,2)$-associativity:
\[
[[abc]de]=[ed[abc]]=[e[dab]c]=[e[bad]c]=[e[bcd]a]=[a[bcd]e].
\]
By Lemma \ref{lm:12assemi}, $(A,[\, ])$ is a knot-theoretic ternary group.
\end{proof}

\section{Geometric application of knot-theoretic ternary groups}\label{S:FVL}

We now use knot-theoretic ternary groups to study curves immersed in compact surfaces. We do it in two ways: first abstractly, via unoriented flat virtual links, and then concretely, by considering curves immersed in compact surfaces that do not need to be orientable, and can have boundary components.

\subsection{Knot-theoretic ternary groups and flat virtual links}
Here we recall basic definitions concerning flat virtual links, for a fundamental introduction to virtual knot theory, see \cite{Kau99}. We only mention that in virtual link diagrams there are both classical, and virtual crossings (denoted with little circles around them). Virtual crossings appear when a diagram placed on a surface (of some genus $g$) is projected on the plane. In flat virtual knot diagrams classical crossings are replaced with their flat version, in which there is no indication which string of the crossing was higher before the projection; such flat crossings appear in the figures in this paper. We now give a formal definition of flat virtual links.

\begin{figure}
\begin{center}
\includegraphics[height=8.5cm]{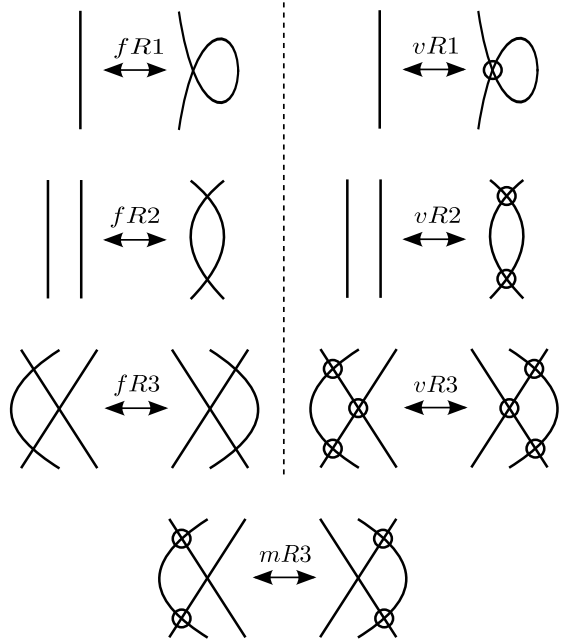}
\caption{Flat virtual moves}\label{flatmoves}
\end{center}
\end{figure}

\begin{de}
A {\it flat virtual link diagram} is a 4-regular plane graph with two types of vertices (also called crossings): flat and virtual.

A {\it flat virtual link} is an equivalence class on flat virtual link diagrams  generated by planar isotopy and the set of moves illustrated in Fig. \ref{flatmoves}. On the left of the figure there are the {\it flat Reidemeister moves}, on the right there are the {\it virtual Reidemeister moves}, and the last move  is the {\it mixed Reidemeister type three move}.
\end{de}

In \cite{NelPic} a strategy was used of utilizing two ternary operations, satisfying mixed axioms, for oriented virtual links. One operation was for the classical crossings, and the other for the virtual ones. Orientation is helpful in distinguishing the four regions around a crossing, however, in our construction we will not assume it. Because of this lack of orientation, we need a simple way of collecting inputs for ternary operations, and this simplicity corresponds to the structure of knot-theoretic ternary groups.

\begin{figure}
\begin{center}
\includegraphics[height=4cm]{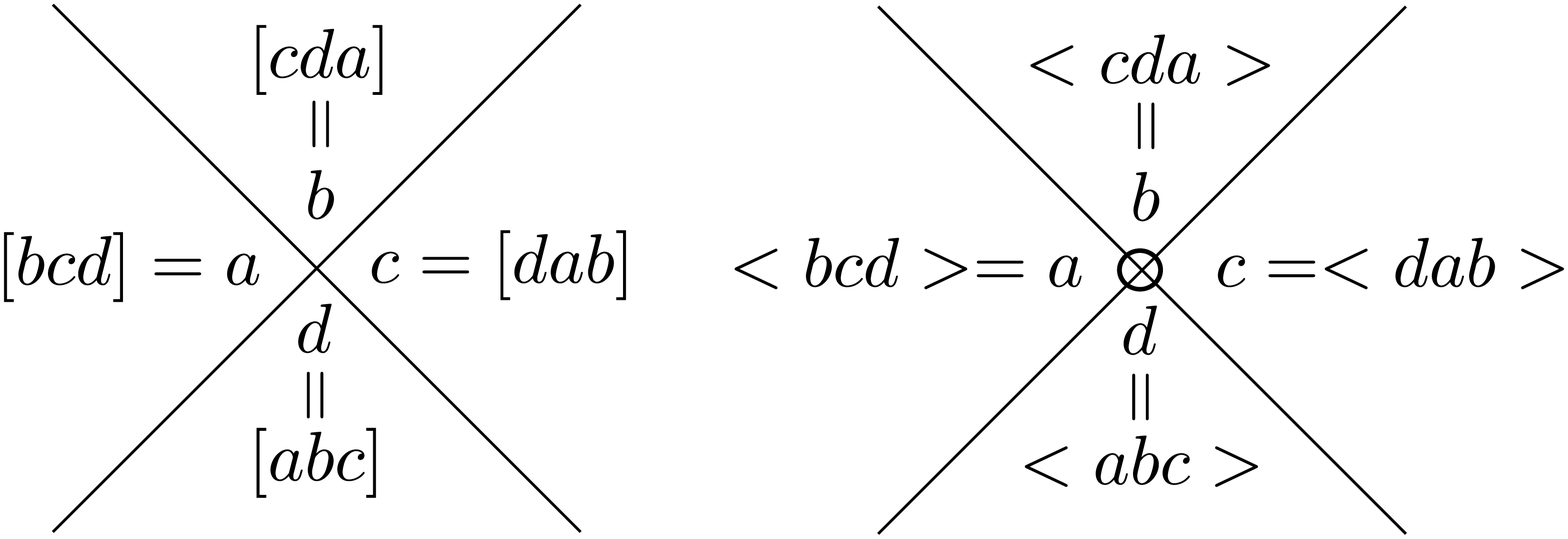}
\caption{Coloring of the regions around flat and virtual crossings}\label{fvcolrule}
\end{center}
\end{figure}

In the invariant, we use two knot-theoretic ternary group structures $(A,[\, ])$ and $(A,<\, >)$ on the same set $A$, the first for the flat crossings and the second for the virtual crossings.
Let $D$ be a diagram of a flat virtual link. Coloring conventions for the flat and virtual crossings are shown in Fig. \ref{fvcolrule}. More specifically, a color is expressed as a ternary product of the other three colors near the crossing taken in a cyclic order; it does not matter if this order is clockwise or counter-clockwise, since any knot-theoretic ternary group is semi-commutative. If there is a consistent way of assigning group elements to the regions in the entire diagram $D$, then we call it a {\it knot-theoretic ternary group coloring} of $D$.
Note that the properties of knot-theoretic ternary groups ensure that the four relations around the crossings in Fig. \ref{fvcolrule} are equivalent, for example:
\begin{align*}
a&=[bcd]\Leftrightarrow b=[a\bar{d}\bar{c}]=[a\bar{\bar{d}}c]=[adc]=[cda],\\
a&=[bcd]\Leftrightarrow c=[\bar{b}a\bar{d}]=[\bar{\bar{b}}ad]=[bad]=[dab],\\
a&=[bcd]\Leftrightarrow d=[\bar{c}\bar{b}a]=[\bar{\bar{c}}ba]=[cba]=[abc].
\end{align*}

\begin{de}
We say that two ternary groups $(A,[\, ])$ and $(A,<\, >)$ are {\it compatible}
if the following condition is satisfied:
\begin{equation}\label{compatible}
[ab<bcd>]=<a<abc>[<abc>cd]>,
\end{equation}
for any $a$, $b$, $c$, $d\in A$.
\end{de}
\begin{example}
Any knot-theoretic ternary group is compatible with itself.
\end{example}

\begin{example}\label{Zkcont}
Let $(A,+)$ be a finite abelian binary group (of an even rank) with a neutral element $0$ and let $x$ be an element of order 2.
Consider two (non-isomorphic) knot-theoretic ternary groups: $(A,[\;])=\T((A,+),0)$ and $(A,<\;>)=\T((A,+),x)$. They
are compatible. Specifically, for any $a$, $b$, $c$, $d\in A$, the condition \eqref{compatible} holds:
\[
[ab<bcd>]=a-b+(b-c+d+x)=a-c+d+x=<acd>,
\]
\[
<a<abc>[<abc>cd]>=a-<abc>+<abc>-c+d+x=a-c+d+x=<acd>.
\]
\end{example}

Now we will show that there are non compatible knot-theoretic ternary groups defined on the same set.
\begin{example}
Let $k>1$ be a natural number. Let $(A,\cdot)\simeq (\Z_{2^k},+_{2^k})$ and $(A,+)\simeq (\Z_2^k,+)$ be two non-isomorphic abelian groups of rank $2^k$ and let $0$ be the neutral element in $(A,+)$, $b$ an element of order 2 in $(A,\cdot)$ and $c$ an element of order 2 in $(A,+)$ such that $c^2\neq 0^2$. Take two knot-theoretic ternary groups
$(A,[\;])=\T((A,\cdot),b)$ and $(A,<\;>)=\T((A,+),c)$.

Then by \eqref{eq:22} and Corollary \ref{cor:count} we have
\[
[0c<ccc>]=[0c(c+ c)]=[0c0]=0\cdot c^{-1}\cdot 0\cdot b=0^2\cdot c^{-1}\cdot b
\]
and
\begin{align*}
& <0<0cc>[<0cc>cc]>=<0(0+ c)[(0+ c)cc]>=\\
&<0(0+ c)((0+ c)\cdot b)>=0+(0+ c)+(0+ c)\cdot b+c=c\cdot b.
\end{align*}
Since, by assumption $c^2\neq 0^2$, the condition \eqref{compatible} does not hold.
\end{example}

\begin{figure}
\begin{center}
\includegraphics[height=4cm]{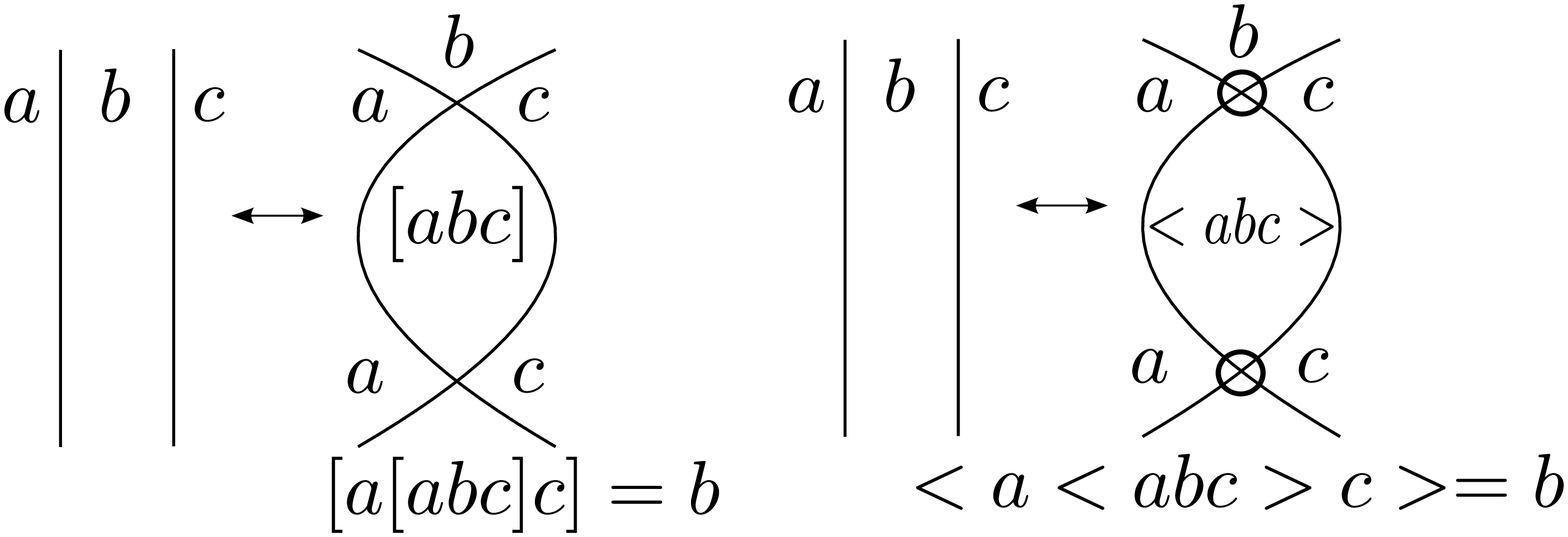}
\caption{Involutions from the second flat and virtual Reidemeister moves}\label{fvR2col}
\end{center}
\end{figure}

\begin{figure}
\begin{center}
\includegraphics[height=6cm]{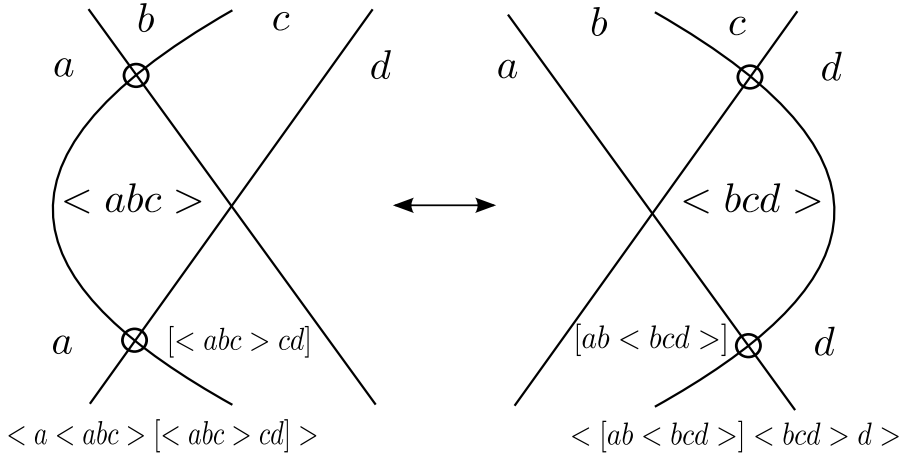}
\caption{A flat mixed move and knot-theoretic ternary group coloring}\label{mixedfv}
\end{center}
\end{figure}

\begin{theorem}\label{cols}
For a pair of compatible knot-theoretic ternary groups $(A,[\, ])$ and
$(A,<\, >)$, and a diagram $D$ of a flat virtual link, the number of knot-theoretic ternary group colorings of $D$ is not changed by the Reidemeister moves used for flat virtual links.
\end{theorem}
\begin{proof}
The argument is analogous for strictly flat moves and strictly virtual moves
(Fig. \ref{flatmoves}). The first flat/virtual move does not change the number of colorings, because the color of the corner in the kink is uniquely determined by the colors of the other three corners of the crossing. The invariance under the second flat/virtual move is a consequence of the involutions $[a[abc]c]=b$
and $<a<abc>c>=b$ that hold in knot-theoretic ternary groups, see Fig. \ref{fvR2col}. The axioms \eqref{A3L} and \eqref{A3R} (for both $[\, ]$ and
$<\, >$) make the number of colorings invariant under the third flat/virtual Reidemeister move, see Fig. \ref{flatR3} for illustration for the operation $[\, ]$. Finally, the invariance under the mixed move is shown in Fig. \ref{mixedfv}.
It requires the condition \eqref{compatible} and the second condition that follows from it by semi-commutativity:
\[
[<abc>cd]=<[ab<bcd>]<bcd>d>.
\]
\end{proof}

\begin{figure}
\begin{center}
\includegraphics[height=5cm]{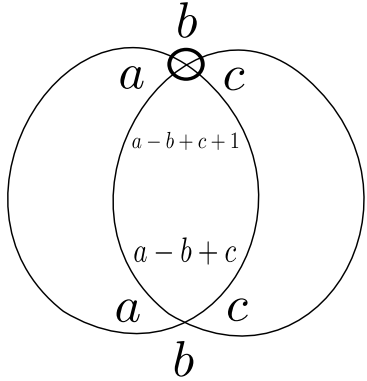}
\caption{The flat-virtual Hopf link is distinguished from the unlink by two-element ternary groups}\label{hopfZ2}
\end{center}
\end{figure}

\begin{example}
Fig. \ref{hopfZ2} illustrates how knot-theoretic ternary group colorings can be used to establish that a given diagram represents a nontrivial flat virtual link. The knot-theoretic ternary groups $(A,[\, ])$ and $(A,<\, >)$ used here are the ones from Example \ref{Zkcont}, more precisely, we take the cyclic group $(\mathbb{Z}_2,+)$ and use the operations
$[abc]=a-b+c \pmod 2$ and $<abc>=a-b+c+1 \pmod 2$. As is clear from Fig. \ref{hopfZ2}, there are no colorings using this pair of compatible knot-theoretic ternary groups.
The number of colorings for the unlink (two disjoint unknotted loops) is $2^3=8$. Thus, the diagram represents a nontrivial flat virtual link.
\end{example}

\subsection{Knot-theoretic ternary groups and curves immersed in compact surfaces}

Let $F$ be a compact surface, possibly with boundary. Let $f\colon\bigsqcup_i S_i^1\sqcup\bigsqcup_k I_k\to F$ be a local embedding from $n$ disjoint circles and $m$ disjoint arcs to $F$ such that precisely the endpoints of the arcs are sent to the boundary components of $F$, if there are any (if not, we only take immersions of circles). Assume that $|f^{-1}f(x)|<3$ for all $x \in \bigsqcup_i S_i^1\sqcup\bigsqcup_k I_k$, that is, we allow at most double points. For simplicity, we let $|f^{-1}f(x)|=1$, if $x$ is an arc endpoint.  We consider images of such maps (diagrams on $F$) up to flat Reidemeister moves and isotopy on F. We let the images of the arc endpoints slide on the boundary of $F$ and pass one through another. We call such an equivalence class a {\it relative flat link} on $F$ if it involves any arcs, or simply a {\it flat link} on $F$ if there are no arcs.

To distinguish (relative) flat link diagrams on a given surface $F$,
we use the same coloring convention as in the previous subsection (Fig. \ref{fvcolrule}). Now, however, only one knot-theoretic ternary group $(A,[\, ])$ is needed, as there are no virtual crossings. We assign elements of $(A,[\, ])$ to the regions of the complement of the (relative) flat link diagram in $F$.
Then we have a theorem analogous to the Theorem \ref{cols}.

\begin{theorem}\label{colsF}
Given a compact surface $F$, a knot-theoretic ternary group $(A,[\, ])$ and a (relative) flat link diagram $D$ on $F$, the number of knot-theoretic ternary group colorings of $D$ is a (relative) flat link invariant.
\end{theorem}

Due to the semi-commutative nature of knot-theoretic ternary groups, we do not require that the surface $F$ be orientable. On a non-orientable surface, we cannot distinguish between ``clockwise" and ``counter-clockwise", but that corresponds exactly to the property $[abc]=[cba]$.

\begin{figure}
\begin{center}
\includegraphics[height=4.5cm]{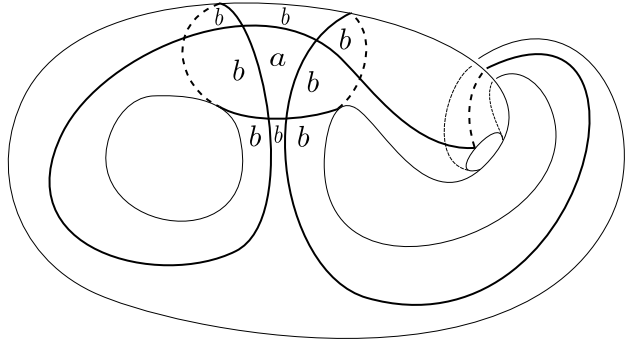}
\caption{The Kishino curve on a non-orientable surface}\label{Kishino}
\end{center}
\end{figure}

\begin{example}
Fig. \ref{Kishino} shows a realization of the flat Kishino knot on a connected sum of a torus and a Klein bottle. There are two connected surface regions in the complement of this curve. Therefore, given a knot-theoretic ternary group  $(A,[\,])=\T((A,+),\bar{e})$, any coloring of this diagram can use at most two colors: $a$ and $b$. There are four flat crossings, but the relations assigned to them are all the same: $a=b-b+b+\bar{e}$. Thus, the color $b$ determines the color $a$, and the number of colorings of the diagram is equal to the order of $(A,[\,])$. This distinguishes the Kishino curve from a trivial (separating) loop on this surface, for which the number of colorings would be $|A|^2$.
\end{example}

\begin{figure}
\begin{center}
\includegraphics[height=4cm]{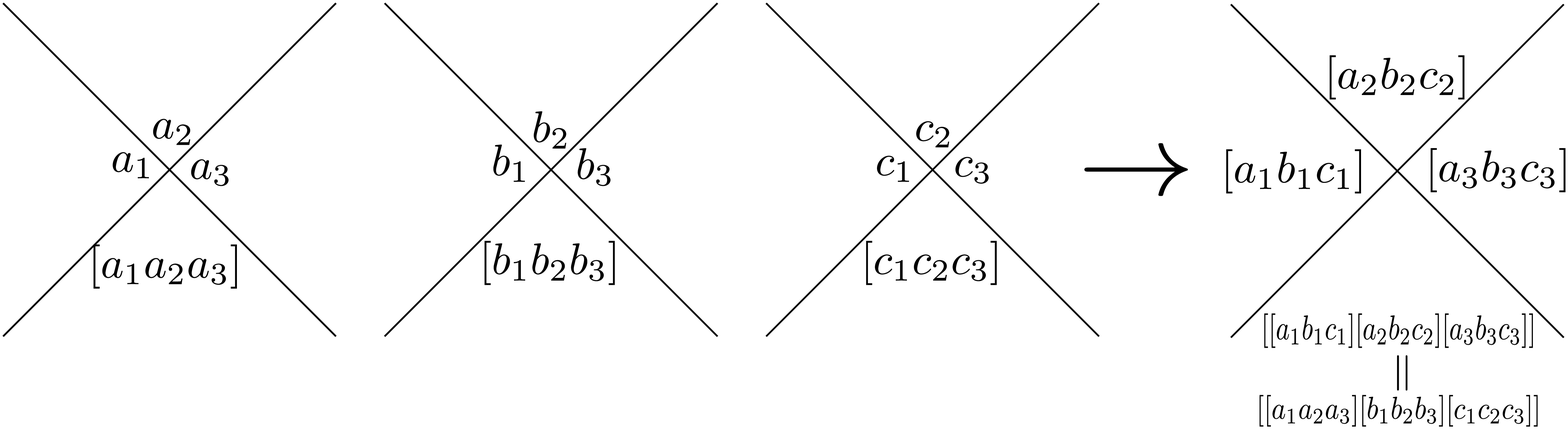}
\caption{Entropic property yields a ternary product on the set of colorings}\label{entropic}
\end{center}
\end{figure}

\begin{remark} The fact that knot-theoretic ternary groups are entropic allows to introduce the knot theoretic ternary group structure on the set of colorings of a given diagram $D$. It is a consequence of the general theory of entropic algebras presented in \cite{RomSmi}. The ternary operation on colorings is performed region-wise, as in Fig. \ref{entropic}. The right side of the figure shows the consistency between the ternary product on colorings and the relation assigned to a crossing, thanks to the entropic condition. The isomorphism class of such knot-theoretic ternary group of colorings is a (relative) flat link invariant.
\end{remark}

\begin{figure}
\begin{center}
\includegraphics[height=9cm]{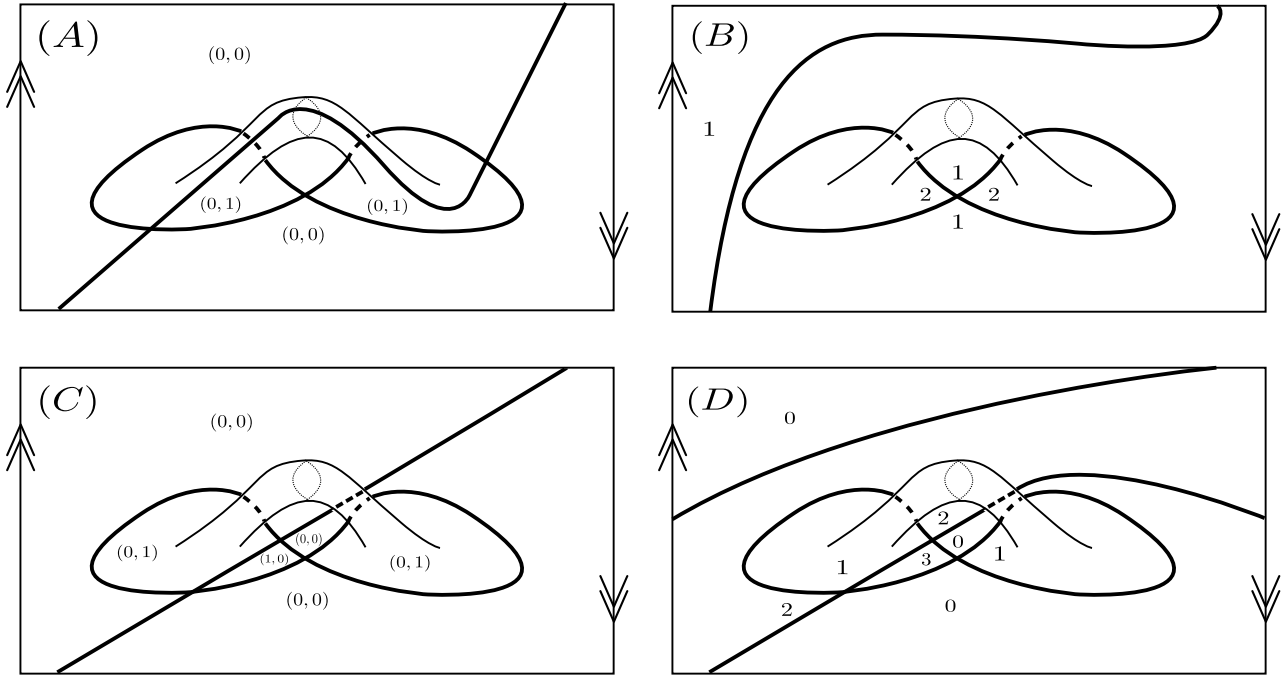}
\caption{Four different relative flat links on a M{\"o}bius band with a handle }\label{Mobiushand}
\end{center}
\end{figure}

\begin{example}
In Fig. \ref{Mobiushand} there are four diagrams of relative flat links, each consisting of an arc and an immersed circle, placed on a M{\"o}bius band with a handle (the sides of the figures are identified according to the arrows). Up to isomorphism, there are four knot-theoretic ternary groups of order four, and we will check how efficient they are in distinguishing these links. We will use $\T((\Z_4,+),0)$, $\T((\Z_4,+),2)$, $\T((\Z_2\times \Z_2,+),(0,0))$ and $\T((\Z_2\times \Z_2,+),(1,1))$. In this order of groups, the numbers of colorings of the diagrams are as follows:
\begin{align*}
&(A):[8,0,16,0],\\
&(B):[8,8,16,0],\\
&(C):[8,8,16,16],\\
&(D):[16,16,16,16].
\end{align*}
Thus, we see that the non-idempotent ternary groups $\T((\Z_4,+),2)$ and $\T((\Z_2\times \Z_2,+),(1,1))$ are sufficient to prove that the four diagrams represent different relative flat links. The figures also show examples of colorings. Fig. \ref{Mobiushand}(A) indicates a coloring with $\T((\Z_2\times \Z_2,+),(0,0))$, and Fig. \ref{Mobiushand}(B) a coloring with $\T((\Z_4,+),2)$.
In Fig. \ref{Mobiushand}(C) $\T((\Z_2\times \Z_2,+),(1,1))$ is used, and as an example we list the involved equations (in order: left, middle, and right crossing):
\begin{align*}
&(0,0)=(0,1)-(1,0)+(0,0)+(1,1),\\
&(1,0)=(0,1)-(0,0)+(0,0)+(1,1),\\
&(1,0)=(0,0)-(0,1)+(0,0)+(1,1).
\end{align*}
Finally, Fig. \ref{Mobiushand}(D) shows a coloring with $\T((\Z_4,+),0)$.
\end{example}

\end{document}